\documentclass[12pt]{amsart}

\usepackage{amsmath}
\usepackage{amssymb}
\usepackage{amsthm}
\usepackage[noadjust]{cite}
\usepackage{mathrsfs}
\usepackage{dsfont}
\usepackage{dcpic}
\usepackage{hyperref}
\usepackage[backgroundcolor=green!40, linecolor=green!40]{todonotes}
\usepackage[all,cmtip]{xy}
\usepackage{xcolor}
\usepackage{yfonts}
\usepackage{enumerate}
\usepackage{bbm}
\usepackage{dsfont}
\usepackage{verbatim}
\usepackage{setspace}
\usepackage{mathtools}

\usepackage[top=1in, bottom=1.25in, left=1.25in, right=1.25in]{geometry}

\hypersetup{colorlinks=true,linkcolor=red,citecolor=blue}

\newtheorem{lemma}{Lemma}[section]
\newtheorem{proposition}[lemma]{Proposition}
\newtheorem{theorem}[lemma]{Theorem}

\newtheorem{conjecture}[lemma]{Conjecture}

\newtheorem*{theorem2}{Theorem}

\theoremstyle{definition}
\newtheorem{definition}[lemma]{Definition}

\newtheorem{hypothesis}[lemma]{Hypothesis}
\newtheorem{remarks}[lemma]{Remarks}

\def\C{{\mathbb C}}

\def\R{{\mathbb R}}

\def\Z{{\mathbb Z}}

\def\Aa{\mathcal{A}}

\def\Dd{\mathcal{D}}

\def\Ii{\mathcal{I}}

\def\Ll{\mathcal{L}}

\def\Vv{\mathcal{V}}

\def\BB{\mathfrak{B}}

\def\SS{\mathfrak{S}}

\def\GGG{\mathbf{G}}
\def\HHH{\mathbf{H}}

\def\LLL{\mathbf{L}}
\def\MMM{\mathbf{M}}

\def\PPP{\mathbf{P}}

\def\SSS{\mathbf{S}}
\def\TTT{\mathbf{T}}
\def\UUU{\mathbf{U}}

\def\XXX{\mathbf{X}}

\def\ZZZ{\mathbf{Z}}
\def\kkk{\mathbf{k}}

\def\Hom{\operatorname{Hom}}

\def\Aut{\operatorname{Aut}}
\def\Ind{\operatorname{Ind}}
\def\cInd{\mathrm{c}\text{-}\operatorname{Ind}}
\def\Res{\operatorname{Res}}
\def\Gal{\operatorname{Gal}}
\def\Irr{\mathrm{Irr}}

\def\GL2{\mathbf{GL}_2}
\def\SL2{\mathbf{SL}_2}
\def\GLN{\mathbf{GL}_N}

\def\SLN{\mathbf{SL}_N}

\def\Rep{\mathrm{Rep}}
\def\rec{\mathrm{rec}}

\newcommand{\dedekind}{{\scriptstyle\mathcal{O}}}

\def\iner{\mathrm{iner}}

\def\dep0{\mathrm{dep\ 0}}

\def\op{\mathrm{op}}
\def\tame{\mathrm{tame}}

\def\ur{\mathrm{ur}}

\def\aff{\mathrm{aff}}

\def\reg{\mathrm{reg}}
\def\inerP{\mathrm{iner}_\Dd}
\def\Frob{\mathrm{Frob}}
\def\ad{\mathrm{ad}}

\makeatletter
\renewcommand\subsection{\@startsection {subsection}{1}{\z@}%
                                        {-3.5ex \@plus -1ex \@minus -.2ex}%
                                        {2.3ex \@plus.2ex}%
                                        {\normalfont\textit}}
\makeatother
\begin{document}
\newtheorem*{recall}{Recall}
\newtheorem*{notat}{Notations}
\newtheorem*{problem}{Problem}
\newtheorem*{fact}{Fact}

\title{The unicity of types for depth-zero supercuspidal representations}
\author{Peter Latham}
\address{Peter Latham\newline Department of Mathematics\newline University of East Anglia\newline Norwich\newline United Kingdom}
\email{peter.latham@uea.ac.uk}
\setlength{\parindent}{0pt}

\begin{abstract}
\noindent We establish the unicity of types for depth-zero supercuspidal representations of an arbitrary $p$-adic group $G$, showing that each depth-zero supercuspidal representation of $G$ contains a unique conjugacy class of typical representations of maximal compact subgroups of $G$. As a corollary, we obtain an inertial Langlands correspondence for these representations via the Langlands correspondence of DeBacker and Reeder.
\end{abstract}
\maketitle

\section{Introduction}

Let $G=\GGG(F)$ be the locally profinite group of rational points of some connected reductive algebraic group $\GGG$ defined over a non-archimedean local field $F$ (for brevity, we call such groups $p$-\emph{adic groups}). The group $G$ naturally acts on a certain Euclidean polysimplicial complex, the \emph{(enlarged) building} of $G$, and the isotropy subgroups of the simplices under this action form important compact open subgroups of $G$. These groups naturally occur as the groups of integral points of certain group subschemes of $\GGG$ which, in general, will not be connected; the groups of integral points of their connected components are known as \emph{parahoric subgroups}. Recall that a (smooth) irreducible complex representation $\pi$ of $G$ is said to be of \emph{depth zero} if it contains a non-zero vector which is fixed by the maximal pro-$p$ normal subgroup of some parahoric subgroup of $G$ (its \emph{pro-unipotent radical}).\\

In \cite{latham2015sl2} and \cite{latham2015sln}, the author described a conjecture on the generalization of the \emph{unicity of types}, previously considered by Henniart and Pa\v{s}k\={u}nas in the case of general linear groups \cite{breuil2002multiplicites,paskunas2005unicity}, to the supercuspidal representations of arbitrary $p$-adic groups. Let us briefly recall how this generalization should look. Given a supercuspidal representation $\pi$ of $G$, a \emph{type} $(J,\lambda)$ for $\pi$ is an irreducible representation $\lambda$ of a compact open subgroup $J$ of $G$ such that an irreducible representation $\pi'$ of $G$ has a non-zero $\lambda$-isotypic component if and only if $\pi'$ is isomorphic to an unramified twist of $\pi$. In the aforementioned papers, we introduced the notion of an \emph{archetype} for $\pi$: these are $G$-conjugacy classes of types for $\pi$ which are defined on maximal compact subgroups of $G$. It is conjectured that each supercuspidal representation $\pi$ of a $p$-adic group $G$ should admit at most one archetype on each conjugacy class of maximal compact subgroups of $G$. The goal of this paper is to establish this result in the case that the representation $\pi$ is of depth-zero.\\

Moy---Prasad and Morris have independently constructed pairs $(G_x,\sigma)$ contained in each depth-zero supercuspidal representation $\pi$ of $G$ \cite{moy1994unrefined,moy1996unrefined,morris1999levelzero}. These pairs are types in a slightly more general sense than that described above: they are types for a finite number of depth-zero supercuspidals, and through a ``refinement'' process, produce types in the above sense. We refer to these types as \emph{unrefined depth-zero types}, and to their refinements as \emph{refined depth-zero types}. Our main technical result is to show that any type which behaves like an unrefined depth-zero type must indeed be isomorphic to one:

\begin{theorem2}
Let $\pi$ be a depth-zero supercuspidal representation of $G$, and let $(G_x,\sigma)$ be an unrefined depth-zero type contained in $\pi$. Let $G_y$ be a maximal parahoric subgroup of $G$, and $\sigma'$ an irreducible representation of $G_y$. Suppose that $(G_y,\sigma')$ is typical for the finite set of supercuspidal representations associated to $\sigma$. Then $(G_y,\sigma')$ is conjugate to $(G_x,\sigma)$.
\end{theorem2}

The point of this result is that it establishes ``unicity on the level of parahoric subgroups''. The proof of this result is completely explicit --- we examine the irreducible subrepresentations of $\pi\downharpoonright_{G_y}$, which fall into two families: the unrefined depth-zero types (which occur when $G_x=G_y$), and the other representations, which we expect to be atypical. For each of the representations $\tau$ in this latter family, we construct a finite set of unrefined depth-zero types defined on a maximal parahoric subgroup of a \emph{proper} Levi subgroup of $G$. Morris has explicitly constructed covers for such types (in the sense of \cite{bushnell1998structure}); we choose an appropriate cover for each of these types, which gives us a finite set of representations of non-maximal parahoric subgroups of $G$, and show that any such representation $\tau$ must intertwine with one of these covers, and hence cannot be typical.\\

While this is a strong result in the direction of unicity, a bit more work is required in order to deduce our conjecture on the unicity of types. If $\GGG$ is semisimple and simply connected as an algebraic group, then the maximal parahoric subgroups of $G$ coincide with the maximal compact subgroups of $G$, in which case the above result reduces precisely to our desired unicity of types result. However, for arbitrary $p$-adic groups there are maximal parahoric subgroups which are not maximal as compact subgroups. The maximal compact subgroups of $G$ all normalize some parahoric subgroup of $G$, although they may not normalize a \emph{maximal} parahoric subgroup. If a maximal compact subgroup $K$ of $G$ does normalize some maximal parahoric subgroup, then $K$ is the maximal compact subgroup of the $G$-normalizer of $K$. We first need to show that one cannot have a type for $\pi$ defined on a maximal compact subgroup of $G$ which does \emph{not} normalize a maximal parahoric subgroup of $G$. Having done so, we assume that the maximal compact subgroup $K$ contains a maximal parahoric $G_x$, and, given a type $(K,\tau)$ for $\pi$, show that the restriction to $G_x$ of $\tau$ is isomorphic to a sum of unrefined depth-zero types. \\

These results are heavily dependent on the explicit nature of our result on the typical representations of parahoric subgroups: they require us to make use of our explicit identification of non-cuspidal representations in which the atypical summands of $\pi\downharpoonright_{G_y}$ are contained. Combining these two results with our unicity result on the level of parahoric subgroups, it is a simple matter to deduce our main theorem:

\begin{theorem2}[The unicity of types for depth-zero supercuspidals]
Let $\pi$ be a depth-zero supercuspidal representation of $G$. The refined depth-zero type associated to $\pi$ by Morris is the unique archetype for $\pi$.
\end{theorem2}

In the previously known cases (of $\GLN(F)$ and $\SLN(F)$ in \cite{breuil2002multiplicites,paskunas2005unicity,latham2015sl2,latham2015sln}), such unicity results have led to instances of the inertial Langlands correspondence. Here, we are also able to obtain an inertial Langlands correspondence, under the assumption that the correspondence constructed by DeBacker and Reeder in \cite{debacker2009depthzero} is the unique correspondence satisfying all of the expected properties of the local Langlands correspondence. In section \ref{DeBacker--Reeder-section}, we give a short summary of this construction. We emphasize that our approach relies heavily on the explicit nature of the DeBacker--Reeder construction. In particular, this means that we are only able to describe the inertial Langlands correspondence for those Langlands parameters which DeBacker and Reeder call \emph{tame regular semisimple elliptic Langlands parameters}.

\begin{theorem2}[The tame inertial Langlands correspondence for supercuspidals]
There exists a unique surjective finite-to-one map from the set of archetypes contained in regular depth-zero supercuspidal representations of $G$ to the set of restrictions to inertia of tame regular Langlands parameters for $G$, which is compatible with the depth-zero Langlands correspondence of DeBacker and Reeder. The fibres of this map admit an explicit description.
\end{theorem2}

The manner in which the fibres of this map are uniformly related to the $L$-packets requires a rather technical setup incorporating all of the main results of this paper; the key observation is that two Langlands parameters have isomorphic restrictions to inertia if and only if the representations in the corresponding $L$-packets contain the same unrefined depth-zero types. We do not elaborate further here; see Theorem \ref{tame-inertial-correspondence}.

\subsection{Acknowledgements}

The results of this paper form part of my UEA PhD thesis; I would like to thank Shaun Stevens for his excellent supervision over the last three years, and EPSRC for funding my studentship. I am also grateful to Jeffrey Adler, Joshua Lansky and Maarten Solleveld for their interest in this work, and for helpful conversations and suggestions regarding previous versions of the paper. Finally, I thank the anonymous referee for many helpful comments.

\section{Notation}

We let $F$ denote a non-archimedean local field, i.e. a locally compact field, complete with respect to a non-archimedean discrete valuation $v_F$ and with finite residue field. Let $\dedekind$ denote the ring of integers of $F$ and $\frak{p}$ the maximal ideal; denote by $\kkk=\dedekind/\frak{p}$ the residue field.\\

By a $p$-\emph{adic group}, we mean the group of rational points of a connected reductive algebraic group defined over $F$, equipped with its natural locally profinite topology. Given a closed subgroup $H$ of a $p$-adic group $G=\GGG(F)$, let $\Res_H^G$ denote the restriction functor from the category of representations of $G$ to the category of representations of $H$. Let $\Ind_H^G$ and $\cInd_H^G$ denote the adjoint functors of induction and compact induction, respectively.\\

We will be concerned with the representation theory of $G$ and its closed subgroups. In this paper, representations of such groups will all be smooth and complex: by a representation of a closed subgroup $H$ of $G$, we mean a complex vector space $V$ and a homomorphism $H\rightarrow\Aut_\C(V)$ such that the stabilizer of each point is an open subgroup of $H$. The collection of all such representations, together with $H$-equivariant morphisms, then forms an abelian category, which we denote by $\Rep(H)$. We write $\Irr(H)$ for the set of isomorphism classes of irreducible objects in $\Rep(H)$; and will occasionally abuse notation by identifying a class in $\Irr(H)$ with a choice of representative.\\

We will denote the left-action of conjugacy by an element $g\in G$ by $^g\cdot$; so if $H$ is a subgroup of $G$ then $^gH=gHg^{-1}$ and, given a representation $\sigma$ of $H$, we write $^g\sigma$ for the representation $^g\sigma(gxg^{-1})=\sigma(x)$ of $^gH$.

\section{Some Bruhat---Tits theory}

We begin by recalling some of the basic constructions from Bruhat---Tits theory which we will need in later sections. Unless specified otherwise, we will follow the constructions from \cite[Section 1]{schneider1997sheaves}.

\subsection{Affine apartments}

We begin by fixing some notation, which will remain in place for the rest of the paper. Let $\GGG$ be a connected reductive group defined over $F$ with Lie algebra $\frak{g}$, and denote by $G=\GGG(F)$ the $F$-rational points of $\GGG$, taken with its natural locally profinite topology. Let $\SSS$ be a maximal $F$-split torus in $\GGG$, which will contain the split component $\ZZZ_\GGG^0$ of the centre of $\GGG$. Let $\TTT$ denote the centralizer in $\GGG$ of $\SSS$; then $\TTT$ is a minimal $F$-Levi subgroup of $\GGG$. Let $\Phi\subset\XXX^*(\SSS)$ be the complete root system for $\GGG$ relative to $\SSS$, and let $\Phi^\aff$ denote the corresponding set of affine roots: the elements of $\Phi^\aff$ are then of the form $\psi=\phi+n$, for some $n\in\Z$, and there is a gradient function $\psi\mapsto\dot{\psi}$ from $\Phi^\aff$ to $\Phi$, given by $\phi+n\mapsto\phi$. Let $W_\GGG$ be the Weyl group of $\GGG$, i.e. $W_\GGG=N_G(\SSS(F))/\TTT(F)$, and let $W_\GGG^\aff=N_G(\SSS(F))/\TTT(\dedekind)$ denote the affine Weyl group. Let $\check{\Phi}\subset\XXX_*(\SSS)$ be the set of coroots dual to $\Phi$, and denote by $\phi\mapsto\check{\phi}$ the natural duality between $\Phi$ and $\check{\Phi}$.\\

Since $\ZZZ_\GGG^0$ is a subtorus of $\SSS$, its cocharacter lattice identifies with a sublattice of the cocharacter lattice of $\SSS$.

\begin{definition}
The \emph{(affine) apartment} associated to the maximal $F$-split torus $\SSS$ of $\GGG$ is the Euclidean space $\mathscr{A}(G,\SSS)=(\XXX_*(\SSS)/\XXX_*(\ZZZ_\GGG^0))\otimes_\Z\R$.
\end{definition}

Note that $\mathscr{A}(G,\SSS)$ naturally inherits the structure of a polysimplicial complex from the underlying quotient lattice.

\subsection{Root subgroups}

To each root $\phi\in\Phi$, we may associate a root subgroup: we let $\UUU_\phi$ denote the unique smooth connected $F$-group subscheme of $\GGG$ such that $\UUU_\phi$ is normalized by $\SSS$ and such that $\mathrm{Lie}(\UUU_\phi)$, together with its natural $\SSS$-action, identifies with the weight space $\frak{g}_\phi+\frak{g}_{2\phi}\subset\frak{g}$. We then write $U_\phi=\UUU_\phi(F)$.\\

The apartment $\mathscr{A}(G,\SSS)$ admits a natural action of $N_G(\SSS(F))$. Indeed, let $v:\ZZZ_\GGG(\SSS)(F)\rightarrow\XXX_*(\SSS)\otimes_\Z\R$ denote the function defined by $\chi(v(z))=-v_F(\chi(z))$ for $z\in\ZZZ_G(F)$ and $\chi\in\XXX^*(\ZZZ_G(\SSS))$; this function extends to a map $N_\GGG(\SSS)(F)\rightarrow\mathrm{Aut}(\mathscr{A}(G,\SSS))$.\\

To each point $x\in\mathscr{A}(G,\SSS)$, we may associate a subgroup of $U_\phi$. A root $\phi\in\Phi$ naturally induces a linear form $\mathscr{A}(G,\SSS)\rightarrow\R$, as well as an involution $s_\phi\in W_\GGG$, which acts on $\mathscr{A}(G,\SSS)$ as $s_\phi\cdot x=x-\phi(x)\check{\phi}$. For each $u\in U_\phi\backslash\{1\}$, the set $U_{-\phi}uU_\phi\cap N_G(\SSS(F))$ then contains a single element $m(u)$ whose image in $W_\GGG$ is $s_\phi$. There exists a real number $l(u)$ such that, for all $x\in\mathscr{A}(G,\SSS)$, one has $m(u)x=s_\phi\cdot x-l(u)\check{\phi}$. This defines a discrete filtration of $U_\phi$ by $U_{\phi,r}=\{u\in U_\phi(F)\backslash\{1\}\ | \ l(u)\geq r\}\cup\{1\}$, where it is to be understood that $U_{\phi,\infty}$ is trivial. Each $x\in\mathscr{A}(G,\SSS)$ then induces a function $f_x:\Phi\rightarrow\R\cup\{\infty\}$ by $\phi\mapsto-\phi(x)$; we let $U_x$ denote the subgroup of $G$ generated by the $U_{\phi,f_x(\phi)}$, as $\phi$ runs through $\Phi$.\\

More generally, to each \emph{affine} root we may associate a similar root subgroup: for $\psi\in\Phi^\aff$, let $U_\psi=\{u\in\UUU_{\dot{\psi}}(F)\ | \ u=1\text{ or } \alpha_{\dot{\psi}}(u)\geq\psi\}$. Here $\alpha_{\dot{\psi}}$ is the unique affine function on $\mathscr{A}(G,\TTT)$ with gradient $\dot{\psi}$ which vanishes on the hyperplane of points fixed by $v(m(u))$.

\subsection{The affine and enlarged buildings}

We now impose the temporary assumption that $\GGG$ is $F$-quasi-split. We define an equivalence relation $\sim$ on $G\times\mathscr{A}(G,\SSS)$ by saying that $(g,x)\sim(h,y)$ if and only if there exists an $n\in N_G(\SSS(F))$ such that $nx=y$ and $g^{-1}hn\in U_x$. Let $\mathscr{B}(G)$ denote the quotient $(G\times\mathscr{A}(G,\SSS))/\sim$.\\

As $\GGG$ is quasi-split over some finite unramified extension, we define the building in general by Galois descent. Let $E/F$ be a finite unramified extension over which $\GGG$ becomes quasi-split. Then the natural action of $\Gal(E/F)$ on $\GGG(E)$ extends to an action on $\mathscr{B}(\GGG(E))$.

\begin{definition}
Let $E/F$ be a finite unramified extension over which $\GGG$ quasi-splits. The \emph{(affine) Bruhat---Tits building} of the $p$-adic group $G=\GGG(F)$ is the space $\mathscr{B}(G)=\mathscr{B}(\GGG(E))^{\Gal(E/F)}$ of $\Gal(E/F)$-fixed points in $\mathscr{B}(\GGG(E))$. 
\end{definition}

This space then carries a natural action of $G$ via left translation.

\begin{theorem}[\cite{bruhat1972buildings,bruhat1984buildings}]
The space $\mathscr{B}(G)$ is a contractible, finite-dimensional, locally finite Euclidean polysimplicial complex on which $G$ acts properly by simplicial automorphisms.
\end{theorem}

Here, the simplices in $\mathscr{B}(G)$ are inherited from those in the apartments $\mathscr{A}(G,\SSS)$. We call such a simplex a \emph{facet}.\\

In the case that $\ZZZ_\GGG(F)$ is not compact, it will be more convenient for us to work with a slight modification of this building. Denote by $\Vv^1$ the dual of $\XXX^*(\GGG)\otimes_\Z\R$.

\begin{definition}
The \emph{enlarged Bruhat---Tits building} of the $p$-adic group $G=\GGG(F)$ is $\mathscr{B}^1(G)=\mathscr{B}(G)\times\Vv^1$. The action of $G$ on $\mathscr{B}(G)$, viewed as the subspace $\mathscr{B}(G)\times\{1\}$ of $\mathscr{B}^1(G)$, extends to an action on $\mathscr{B}^1(G)$ via $g\cdot(x,v)=(g\cdot x,v+\theta(g))$, where $\theta(g)(\chi)=-v_F(\chi(g))$.
\end{definition}

Say that a \emph{facet} of $\mathscr{B}^1(G)$ is a subspace of the form $V\times\Vv^1$, where $V\subset\mathscr{B}(G)$ is a facet in the above sense. We adopt the convention that facets in $\mathscr{B}(G)$ are closed --- the reader should note that some other authors take the facets to be the interior of these simplices (with the interior of a vertex being the vertex itself). We will always make it clear when we wish only to consider the interior. While a point $x\in\mathscr{B}(G)$ may be contained in multiple facets, there will always exist a unique facet of minimal dimension which contains $x$; we denote the facet in $\mathscr{B}^1(G)$ associated to this facet by $\bar{x}$.\\

Given a vertex $x\in\mathscr{B}(G)$, we say that the \emph{link} of $x$ is the union of the facets in which $x$ is contained.

\subsection{Parahoric group schemes}

\begin{proposition}[\cite{bruhat1972buildings,bruhat1984buildings}]
Let $x\in\mathscr{B}(G)$.
\begin{enumerate}[(i)]
\item The $G$-isotropy subgroup $\tilde{G}_x$ of $\bar{x}\subset\mathscr{B}^1(G)$ is a compact open subgroup of $G$.
\item There exists a unique smooth affine $\dedekind$-group subscheme $\tilde{\GGG}_x$ of $\GGG$ with generic fibre $\GGG$ such that $\tilde{\GGG}_x(\dedekind)=\tilde{G}_x$.
\end{enumerate}
\end{proposition}

In general, the schemes $\tilde{\GGG}_x$ will \emph{not} be connected. The special fibre of $\tilde{\GGG}_x$ will identify with a reductive $\kkk$-group scheme; the inverse image under the projection onto the special fibre of the connected component of this $\kkk$-group scheme is then a smooth connected $\dedekind$-group subscheme of $\tilde{\GGG}_x$, which we denote by $\GGG_x$. We call the connected group scheme $\GGG_x$ the \emph{parahoric group scheme} associated to $x$. Its group $G_x=\GGG_x(\dedekind)$ of $\dedekind$-rational points is then a compact open subgroup of $G$, which we call the \emph{parahoric subgroup} of $G$ associated to $x$.

\begin{proposition}[\cite{bruhat1972buildings,bruhat1984buildings}]
The group $G_x$ contains a unique maximal pro-$p$ normal subgroup $G_x^+$, which identifies with the group of $\dedekind$-rational points of the pro-unipotent radical of $\GGG_x$. The quotient $G_x/G_x^+$ is isomorphic to the group of $\kkk$-rational points of the connected component of the special fibre of $\tilde{\GGG}_x$, which is a finite reductive group over $\kkk$.
\end{proposition}

Moreover, we may \emph{explicitly} describe the groups $G_x$ and $G_x^+$ in terms of generators and relations (see, for example, \cite[2.3]{fintzen2015stable}):

\begin{proposition}\label{generators-of-parahorics}
Let $x\in\mathscr{B}(G)$. Then one has:
\begin{enumerate}[(i)]
\item $G_x=\langle \TTT(F)\cap G_x,U_\psi\ | \ \psi\in\Phi^\aff,\ \phi(x)\geq 0\rangle$; and
\item $G_x^+=\langle \TTT(F)\cap G_x,U_\psi\ | \ \psi\in\Phi^\aff,\ \psi(x)> 0\rangle$.
\end{enumerate}
\end{proposition}

One also has an order-reversing bijection between the set of parahoric subgroups of $G$ and the simplicial skeleton of $\mathscr{B}(G)$:

\begin{proposition}[\cite{bruhat1972buildings,bruhat1984buildings}]
Let $x,y\in\mathscr{B}(G)$. Then one has $G_x\subset G_y$ if and only if $\bar{y}\subset\bar{x}$.
\end{proposition}

In particular, the parahoric subgroup $G_x$ does not depend on the choice of point $x$ in the interior of the facet containing $x$. In the case that $x$ is contained in the interior of a facet of minimal dimension (a \emph{chamber}), the parahoric subgroup $G_x$ is a minimal parahoric subgroup, which we call an \emph{Iwahori subgroup}.\\

In a closely related manner, we may also associate to the interior of each facet in $\mathscr{B}(G)$ an $F$-Levi subgroup of $G$. The intuition behind this is as follows: if $G_y\subset G_x$ is a strict containment of parahoric subgroups of $G$, then the image in $G_x/G_x^+$ of $G_y$ identifies with a proper parabolic subgroup of $G_x/G_x^+$. The Levi subgroup $M$ we construct will then be precisely the one such that the image in $G_x/G_x^+$ of $M\cap G_x$ identifies with the standard Levi factor of the image of $G_y$.

\begin{proposition}[{\cite[Proposition 6.4]{moy1996unrefined}}]\label{existence-of-levi}
Let $x\in\mathscr{B}(G)$. The algebraic subgroup $\MMM$ of $\GGG$ generated by $\TTT$, together with the root subgroups $\UUU_\phi$ for those $\phi\in\Phi$ such that some affine root $\phi+n$, $n\in\Z$, vanishes on $\bar{x}$, is an $F$-Levi subgroup of $G$. The group $\MMM(F)\cap G_x$ is a maximal parahoric subgroup of $\MMM(F)$, and there is a natural identification $(\MMM(F)\cap G_x)/(\MMM(F)\cap G_x)^+=G_x/G_x^+$.
\end{proposition}

\section{Depth-zero types}

\subsection{Types}

Let us begin by recalling the definition of a type, for which it will be convenient for us to recall the Bernstein decomposition of the category $\Rep(G)$ of smooth complex representations of $G$. Say that a \emph{cuspidal datum} in $G$ is a pair $(M,\zeta)$ consisting of an $F$-Levi subgroup $M$ of $G$ and a supercuspidal representation $\zeta$ of $M$. We say that two cuspidal data $(M,\zeta)$ and $(M',\zeta')$ are \emph{inertially equivalent} if there exists an unramified character $\omega$ of $M$ such that $(M',\zeta')$ is $G$-conjugate to $(M,\zeta\otimes\omega)$. Let $\BB(G)$ denote the set of inertial equivalence classes of cuspidal data, with $[M,\zeta]_G$ denoting the class of $(M,\zeta)$. An irreducible representation $\pi$ of $G$ has a unique (up to conjugacy) \emph{supercuspidal support}, i.e. a unique $G$-conjugacy class $(M,\zeta)$ of cuspidal data such that $\pi$ is isomorphic to a subquotient of the parabolically induced representation $\Ind_{M,P}^G\ \zeta$, where $P$ is some $F$-parabolic subgroup of $G$ with Levi factor $M$. The \emph{inertial support} of $\pi$ is the inertial equivalence class of its supercuspidal support.\\

To each $\frak{s}\in\BB(G)$, we associate a full subcategory $\Rep^{\frak{s}}(G)$ of $\Rep(G)$ consisting of those representations all of whose irreducible subquotients are of inertial support $\frak{s}$.

\begin{theorem}[\cite{bernstein1984centre}]
The subcategories $\Rep^{\frak{s}}(G)$ are indecomposable, and one has 
\[\Rep(G)=\prod_{\frak{s}\in\BB(G)}\Rep^{\frak{s}}(G).
\]
\end{theorem}

More generally, given a subset $\SS$ of $\BB(G)$, let $\Rep^\SS(G)=\prod_{\frak{s}\in\SS}\Rep^{\frak{s}}(G)$. A type then allows one to describe these indecomposable subcategories $\Rep^{\frak{s}}(G)$ in terms of the representation theory of compact open subgroups of $G$:

\begin{definition}
Let $(J,\lambda)$ be a pair consisting of a smooth irreducible representation $\lambda$ of a compact open subgroup $J$ of $G$, and let $\SS\subset\BB(G)$.
\begin{enumerate}[(i)]
\item We say that $(J,\lambda)$ is $\SS$-\emph{typical} if, for any irreducible representation $\pi'$ of $G$, one has that $\Hom_J(\pi'\downharpoonright_J,\lambda)\neq 0\Rightarrow\pi'\in\Rep^{\SS}(G)$.
\item We say that $(J,\lambda)$ is an $\SS$-\emph{type} if it is $\SS$-typical and, given any irreducible representation $\pi'$ in $\Rep^{\SS}(G)$, one has $\Hom_J(\pi'\downharpoonright_J,\lambda)\neq 0$.
\end{enumerate}
In the case that $\SS=\{\frak{s}\}$ is a singleton, we simply speak of $\frak{s}$-types.
\end{definition}

We will require a slight modification of this definition: the notion of an \emph{archetype}, which the author introduced in \cite{latham2015sl2} and \cite{latham2015sln}; these two papers elaborate further on the motivation behind the definition.

\begin{definition}
Let $\frak{s}\in\BB(G)$. An $\frak{s}$-\emph{archetype} is a $G$-conjugacy class of $\frak{s}$-typical representations of a maximal compact subgroup of $G$.
\end{definition}

We will often abuse terminology and simply speak of a single pair $(K,\tau)$ with $K$ maximal compact in $G$ as being an archetype; it should always be understood that we are considering questions modulo conjugacy (as it is easy to see that any conjugate of a type is still a type).\\

In the aforementioned papers, the author suggested the following conjecture:

\begin{conjecture}[The unicity of types for supercuspidals]\label{unicity-conjecture}
Let $\pi$ be a supercuspidal representation of $G$. Then there exists a $[G,\pi]_G$-archetype, and there exists at most one $[G,\pi]_G$-archetype defined on each $G$-conjugacy class of maximal compact subgroups of $G$.
\end{conjecture}

This conjecture then naturally raises the question of on \emph{which} conjugacy classes of maximal compact subgroups there exists a $[G,\pi]_G$-type. This however, seems to still be beyond understanding for the present, with the answer only known for $\GLN(F)$ (where there is a single conjugacy class) and $\SLN(F)$ \cite{latham2015sl2,latham2015sln}.\\

Our goal in this paper is to establish this conjecture in the case that $\pi$ is of depth-zero. Recall that we say that an irreducible representation $\pi$ of $G$ is of depth-zero if it has a non-zero vector fixed by the pro-unipotent radical of some parahoric subgroup of $G$. We also answer the question of how many archetypes are contained in such a representation: there will turn out to be precisely one.

\subsection{Unrefined depth-zero types}

In \cite{moy1994unrefined,moy1996unrefined,morris1999levelzero}, Moy---Prasad and Morris construct natural conjugacy classes of types for depth-zero supercuspidals $\pi$. We briefly recall these constructions.

\begin{definition}
An \emph{(unrefined) depth-zero type} in $G$ is a pair $(G_x,\sigma)$ consisting of a parahoric subgroup $G_x$ of $G$ and an irreducible cuspidal representation $\sigma$ of $G_x/G_x^+$.
\end{definition}

The etymology of these depth-zero types is due to the following:

\begin{theorem}[{\cite[Theorem 4.5]{morris1999levelzero}}]
Let $(G_x,\sigma)$ be an unrefined depth-zero type in $G$. Then there exists a finite set $\SS_\sigma\subset\BB(G)$ such that $(G_x,\sigma)$ is an $\SS_\sigma$-type. Any depth-zero irreducible representation $\pi$ of $G$ contains a unique $G$-conjugacy class of unrefined depth-zero types. Moreover, if $G_x$ is a maximal parahoric subgroup of $G$, then $\Irr^{\SS_\sigma}(G)$ consists only of supercuspidal representations.
\end{theorem}

Morris also shows that there are natural relations between the unrefined depth-zero types in $G$ and those in its Levi subgroups: certain of the unrefined depth-zero types in $G$ are covers of those defined on a \emph{maximal} parahoric subgroup of $M$, in the sense of \cite{bushnell1998structure}. We will require a slightly different formulation to that given by Morris in the non-cuspidal case. Unravelling Morris' approach, he takes an $F$-Levi subgroup $M$ of $G$ and an unrefined depth-zero type $(M_x,\sigma)$ in $M$, to which one associates a finite set $\SS_\sigma\subset\BB(M)$ as above. This set then defines a finite subset $\SS_\sigma'$ of $\BB(G)$: the set $\SS_\sigma'$ consists of the inertia classes $[M,\zeta]_G$, for those $[M,\zeta]_M\in\SS_\sigma$. Morris then chooses an embedding $\mathscr{B}(M)\hookrightarrow\mathscr{B}(G)$ which maps the vertex $x\in\mathscr{B}(M)$ into some facet of positive dimension in $\mathscr{B}(G)$ which is associated to $M$ by Proposition \ref{existence-of-levi}. To this facet corresponds a non-maximal parahoric subgroup $J$ of $G$, and the representation $\sigma$ of $M_x\subset J$ extends by the trivial character to a representation of $J$; the resulting pair $(J,\sigma)$ is then a $G$-cover of $(M_x,\sigma)$. We may use Morris' result in the following form:

\begin{theorem}[{\cite[Theorem 4.8]{morris1999levelzero}}]\label{depth-zero-covers}
Let $x\in\mathscr{B}(G)$, and let $M=\MMM(F)$ be the $F$-Levi subgroup of $G$ associated to $\bar{x}$. Let $x_M\in\mathscr{B}(M)$ be such that $M_{z_M}=M\cap G_x$, and let $(M_{x_M},\sigma)$ be a depth-zero type in $M$. Choose an embedding $j_M:\mathscr{B}(M)\hookrightarrow\mathscr{B}(G)$ such that $\overline{j_M(x_M)}=\bar{x}$. Then there exists a $G$-cover of $(M_{x_M},\sigma)$ which is of the form $(G_{j_M(x_M)},\lambda)$, where $\lambda\downharpoonright_{M_{x_M}}=\sigma$. Moreover, the group $G_{j_M(x_M)}$ has an Iwahori decomposition with respect to any parabolic subgroup of $G$ with Levi factor $M$. Moreover, $\lambda$ is trivial on the upper- and lower-unipotent subgroups complementary to $M$ in the Levi decompositions of parabolics containing $M$.
\end{theorem}

The significance of this cover is that it will be an $\SS_\sigma'$-type \cite[Theorem 4.5]{morris1999levelzero}. In particular, any irreducible subquotient of $\cInd_{G_{j_M(x_M)}}^G\ \lambda$ must be non-cuspidal.

\subsection{Refinement of depth-zero types}

Given an unrefined depth-zero type $(G_x,\sigma)$, we wish to ``refine'' this type in order to obtain an $\frak{s}$-type for each $\frak{s}\in\SS_\sigma$. Let $K$ denote the maximal compact subgroup of the $G$-normalizer of $G_x$; if $G_x$ is a maximal parahoric subgroup of $G$ then $K$ is a maximal compact subgroup of $G$ which contains $G_x$ as a normal subgroup of finite index. Note that one actually has an identification $K=\tilde{\GGG}_x(\dedekind)$.

\begin{theorem}[{\cite[Theorem 4.7]{morris1999levelzero}}]\label{morris-refinement}
Suppose that $G_x$ is a maximal parahoric subgroup of $G$. For each irreducible subrepresentation $\tau$ of $\Ind_{G_x}^K\ \sigma$, there exists an $\frak{s}\in\SS_\sigma$ such that $(K,\tau)$ is an $\frak{s}$-type. Conversely, for every $\frak{s}\in\SS_\sigma$, there exists an $\frak{s}$-type of this form.
\end{theorem}

Thus, for each depth-zero supercuspidal representation $\pi$ of $G$, we have a construction of a $[G,\pi]_G$-archetype. Our goal is to show that these exhaust all such archetypes.

\section{Unicity on the level of parahoric subgroups}\label{section:parahoric-unicity}

We first establish a partial result, which may be viewed as a unicity result on the level of unrefined types. This is our main technical result, and the proof will occupy the remainder of this section.

\begin{theorem}\label{unrefined-unicity}
Let $\pi$ be a depth-zero supercuspidal representation of $G$, and let $(G_x,\sigma)$ be an unrefined depth-zero type contained in $\pi$. Suppose that $y$ is a vertex in $\mathscr{B}(G)$ such that there exists some $\SS_\sigma$-typical representation $\sigma'$ of $G_y$. Then $(G_y,\sigma')$ is $G$-conjugate to $(G_x,\sigma)$.
\end{theorem}

Since our approach will rely on identifying explicit relations among the parahoric subgroups of $G$, we begin by fixing a notion of a standard parahoric subgroup. Fix, once and for all, a chamber $X\subset\mathscr{B}(G)$, and let $i$ be an element of the interior of $X$; thus $I=G_i$ is an Iwahori subgroup of $G$. We refer to $X$ as the \emph{standard chamber} of $\mathscr{B}(G)$. Let $P_\emptyset$ denote a parabolic subgroup of $G$ with Levi factor the Levi subgroup associated to $i$ and such that $G_i=(P_\emptyset\cap G_i)G_i^+$, so that $P_\emptyset$ is a minimal $F$-parabolic subgroup of $G$. We say that a parabolic subgroup $P$ of $G$ is standard if $P\supset P_\emptyset$, and that a parahoric subgroup $G_x$ of $G$ is standard if $x\in X$.

\begin{proof}[Proof of Theorem \ref{unrefined-unicity}]
Since every parahoric subgroup of $G$ is conjugate to a standard parahoric subgroup, $\pi$ contains a depth-zero type $(G_x,\sigma)$ such that $G_x$ is standard. Our claim is then that, given any maximal parahoric subgroup $G_y$ of $G$, the irreducible subrepresentations $\tau$ of $\pi\downharpoonright_{G_y}$ are atypical for $\SS_\sigma$, unless $G_x$ is conjugate to $G_y$ and $\tau$ is conjugate to $\sigma$. Alternatively, it suffices to check this for the \emph{standard} maximal parahoric subgroups $G_y$.\\

Certainly, $\pi$ appears as a subquotient of $\cInd_{G_x}^G\ \sigma$, and so we may embed $\pi\downharpoonright_{G_y}$ into the Mackey decomposition of $\Res_{G_y}^G\cInd_{G_x}^G\ \sigma$:
\[\pi\downharpoonright_{G_y}\hookrightarrow\bigoplus_{g:G_y\backslash G/G_x}\Ind_{K_g}^{G_y}\Res_{K_g}^{^gG_y}\ ^g\sigma,
\]
where we write $K_g={}^gG_x\cap G_y$. Any double coset in the space $G_x\backslash G/G_y$ admits a representative in $W_\GGG^\aff$; we will always assume that the representative $g$ is such a representative \emph{which is of shortest length} (relative to the standard involutions generating $W_G^\aff$). In particular, this guarantees the following:

\begin{lemma}[{\cite[Lemma 3.19, Corollary 3.20]{morris1993intertwining}}]
If $g$ is a shortest length coset representative for $G_x\backslash G/G_y$ in $W_\GGG^\aff$ and either $G_x\neq G_y$ or $g\not\in N_G(G_x)$, then the image in $G_y/G_y^+$ of $K_g$ is a proper parabolic subgroup $\mathscr{P}$ of $G_y/G_y^+$. The preimage in $G_y$ of $\mathscr{P}$ is $G_y^+K_g$, and there exists a point $z$ in the link of $y$ such that $G_z=G_y^+K_g$.
\end{lemma}

Let~$\tau$ be an irreducible subrepresentation of~$\pi\downharpoonright_{G_y}$, and suppose that~$\tau$ is a subrepresentation of $\Ind_{K_g}^{G_y}\Res_{K_g}^{^gG_x}\ ^g\sigma$ for some $g$ satisfying the above hypotheses (if $G_x=G_y$ and $g\in N_G(G_x)$, then $\tau$ is necessarily a conjugate of $\sigma$). Then there exists an irreducible subrepresentation $\Xi$ of $\Res_{K_g}^{^gG_x}\ ^g\sigma$ such that $\tau\hookrightarrow\Ind_{K_g}^{G_y}\ \Xi$. Note that $\Xi$ is clearly trivial on $^gG_x^+\cap G_y$.\\

Let $M=\MMM(F)$ denote the $F$-Levi subgroup of $G$ associated to $z$ by Proposition \ref{existence-of-levi}, and let $P=\PPP(F)$ denote the standard parabolic subgroup of $G$ with Levi factor $M$. From Proposition \ref{existence-of-levi}, we obtain the following:

\begin{lemma}
There exists a vertex $z_M\in\mathscr{B}(M)$ such that $M_{z_M}=M\cap G_z$. One has $M_{z_M}^+=M\cap G_z^+=M\cap G_y^+$, which induces a natural identification $M_{z_M}/M_{z_M}^+=\mathscr{M}$.
\end{lemma}

We now need to choose an embedding of $\mathscr{B}(M)$ into $\mathscr{B}(G)$, which comes down to choosing the image of $z_M$. We let $j_M:\mathscr{B}(M)\hookrightarrow\mathscr{B}(G)$ be an embedding which maps $z_M$ into the link of $g\cdot x$ in a way that $g\cdot x$ lies on the unique geodesic $\gamma$ from $j_M(z_M)$ to $y$, and such that $M\cap G_{j_M(z_M)}=M_{z_M}$. Note that, while such an embedding exists, it is clearly not unique. It is, however, essentially unique for our purposes: its restriction to the link of $z_M$ in $\mathscr{B}(M)$ is unique up to the image in $\gamma\cap\overline{j_M(z_M)}\cap\mathrm{link}(g\cdot x)$ of $z_M$, which has no effect on the parahoric subgroups which will be defined via this embedding.\\

With this in place, we are ready to begin our examination of the representation $\tau$. Since $M_{z_M}^+\subset{}^gG_x^+$ (for the same reason that $M_{z_M}^+\subset G_y^+$), we see that $M_{z_M}^+\subset{}^gG_x^+\cap G_y$. As $\Xi$ has already been noted to act trivially on this group, we deduce that:

\begin{lemma}
The representation $\Xi$ is trivial on $M_{z_M}^+$.
\end{lemma}

So, in particular, the representation $\Xi\downharpoonright_{M_{z_M}}$ identifies with a representation of $M_{z_M}/M_{z_M}^+=\mathscr{M}$. Given an irreducible subrepresentation $\xi$ of $\Xi\downharpoonright_{M_{z_M}}$, we therefore have a notion of the cuspidal support of $\xi$, as the unique $\mathscr{M}$-conjugacy class of pairs $(\mathscr{L}_\xi,\zeta_\xi)$ of cuspidal representations of Levi subgroups of $\mathscr{M}$ such that $\xi$ is contained in the representation parabolically induced from $(\mathscr{L}_\xi,\zeta_\xi)$. Let $\mathscr{Q}_\xi$ be a parabolic subgroup of $\mathscr{M}$ with Levi factor $\mathscr{L}_\xi$, standard in the sense that it contains the image in $\mathscr{M}$ of our fixed Iwahori subgroup $I$ of $G$ (which, upon intersection with $M$, gives an Iwahori subgroup of $M$). The following is then a simple observation:

\begin{lemma}
The inverse image in $G_y$ of $\mathscr{Q}_\xi$ is a standard parahoric subgroup of $G$ corresponding to some point $w$ in the standard chamber of $\mathscr{B}(G)$, and one has a containment $G_w\subset G_z$.
\end{lemma}

Indeed, this inverse image must certainly be a parahoric subgroup of $G$ contained in $G_z$. Moreover, since it contains the inverse image of the minimal Levi in $G_y/G_y^+$ corresponding to our fixed Iwahori subgroup of $G$, it contains this Iwahori subgroup, and so corresponds to a point in the standard chamber.\\

As before, let $L=\LLL(F)$ denote the $F$-Levi subgroup associated to the point $w\in\mathscr{B}(G)$, and let $w_L\in\mathscr{B}(L)$ be a vertex such that $L_{w_L}=L\cap M_{z_M}$. Choose some embedding $\iota:\mathscr{B}(L)\hookrightarrow\mathscr{B}(M)$ such that $L_{w_L}=L\cap M_{\iota(w_L)}$. The embedding $j_M:\mathscr{B}(M)\hookrightarrow\mathscr{B}(G)$ then gives us an embedding $j_L=j_M\circ\iota:\mathscr{B}(L)\hookrightarrow\mathscr{B}(G)$. With this in place, by Theorem \ref{depth-zero-covers}, we are able to construct a $G$-cover $(J_\xi,\lambda_\xi)$ of the depth-zero type $(L_{w_L},\zeta_\xi)$ such that $J_\xi=G_{j_L(w_L)}$.\\

We begin by considering, for each fixed choice of irreducible representation $\xi$ of $M_{z_M}/M_{z_M}^+$ as above, the space
\begin{align*}
\Hom_G(\cInd_{J_\xi}^G\ \lambda_\xi,\cInd_{K_g}^G\ \Xi)&=\Hom_{J_\xi}(\lambda_\xi,\Res_{J_\xi}^G\cInd_{K_g}^G\ \Xi)\\
&=\bigoplus_{J_\xi\backslash G/K_g}\Hom_{J_\xi}(\lambda_\xi,\Ind_{J_\xi\cap{}^hK_g}^{J_\xi}\Res_{J_\xi\cap{}^hK_g}^{^hK_g}\ ^h\Xi)\\
&=\bigoplus_{J_\xi\backslash G/K_g}\Hom_{J_\xi\cap{}^hK_g}(\lambda_\xi,{}^h\Xi).
\end{align*}
This space then surjects onto the summand corresponding to $h=1$, namely onto $\Hom_{J_\xi\cap K_g}(\lambda_\xi,\Xi)$.\\

Since $\lambda_\xi$ is an extension of $\zeta_\xi$ by the trivial character of the unipotent subgroups in the Iwahori decomposition of $G_{j_L(w_L)}$ with respect to $L$, it must certainly be the case that $\lambda_\xi|_{J_\xi\cap K_g}$ is trivial on the upper and lower unipotent parts of the Iwahori decomposition of $G_{j_M(z_M)}$ with respect to the standard parabolic subgroup of $G$ with Levi factor $M_{z_M}$. This means that, if we knew that $J_\xi\cap K_g\subset M_{z_M}{}^gG_x^+$, then we would be able to make the identification
\[\Hom_{J_\xi\cap K_g}(\lambda_\xi,\Xi)=\Hom_{J_\xi\cap M_{z_M}}(\lambda_\xi,\Xi),
\]
where the latter space is clearly non-zero due to the construction of $\lambda_\xi$ as a cover of the cuspidal support of an irreducible subrepresentation of $\Xi|_{M_{z_M}}$. So it remains to check that $J_\xi\cap K_g\subset M_{z_M}{}^gG_x^+$. Since $J_\xi=G_{j_L(w_L)}\subset G_{j_M(z_M)}$ and $^gG_x=G_{gx}$, it suffices to check that any of the generators obtained from Proposition \ref{generators-of-parahorics} for the group $G_{j_M(z_M)}\cap K_g=G_{j_M(z_M)}\cap G_{gx}\cap G_y$ are contained in $M_{z_M}G_{gx}^+$. This is straightforward.\\

Now we return to the representation $\tau$. Since, as $\xi$ ranges over the irreducible subrepresentations of $\Xi|_{M_{z_M}}$, the images in $\Xi$ of the elements of $\Hom_{K_g}(\bigoplus_\xi \Ind_{J_\xi}^{K_g}\ \lambda_\xi,\Xi)$ generate $\Xi$, composing the non-zero maps $\cInd_{J_\xi}^G\ \lambda_\xi\rightarrow\cInd_{K_g}^G\ \Xi$ and $\cInd_{K_g}^G\ \Xi\rightarrow\cInd_{G_y}^G\ \tau$ (with the latter arising from $\tau\rightarrow\Ind_{K_g}^{G_y}\ \Xi$) results in, for some $\xi$, a non-zero map $\cInd_{J_\xi}^G\ \lambda_\xi\rightarrow\cInd_{G_y}^G\ \tau$.\\

So, by Frobenius reciprocity, the representation $\tau$ is contained in some irreducible subquotient of $\cInd_{J_\xi}^G\ \lambda_\xi$, for some irreducible subrepresentation $\xi$ of $\Xi|_{M_{z_M}}$. Since $\lambda_\xi$ is a $G$-cover of $(L_\xi,\zeta_\xi)$, any such irreducible subquotient must be non-cuspidal, and so $\tau$ is contained in some non-cuspidal irreducible representation of $G$, and hence cannot be $\SS_\sigma$-typical.
\end{proof}

In the next section, it will be convenient for us to have a slight generalization of this result, showing that there do not exist any $\SS_\sigma$-types defined on non-maximal parahoric subgroups of $G$:

\begin{proposition}\label{all-parahorics-unicity}
Let $y\in\mathscr{B}(G)$, and let $\pi$ be a depth-zero supercuspidal representation of $G$. Let $(G_x,\sigma)$ be an unrefined depth-zero type contained in $\pi$. If $y$ is not a vertex, then no irreducible subrepresentation of $\pi|_{G_y}$ may be $\SS_\sigma$-typical.
\end{proposition}

\begin{proof}
Suppose that $y$ is not a vertex. Since $y$ is contained in the interior of a facet of positive dimension, there exists a vertex $z\in\mathscr{B}(G)$ such that $G_y\subset G_z$. Let $\Xi$ be an irreducible subrepresentation of $\pi|_{G_y}=\pi|_{G_z}|_{G_y}$. By Theorem \ref{unrefined-unicity}, unless $z$ is conjugate to $x$ under the action of $G$, we may find a non-cuspidal irreducible representation $\pi'$ of $G$ in which $\Xi$ is contained, which shows that $\Xi$ is atypical for $\SS_\sigma$. So, without loss of generality, we assume that $G_y\subset G_x$. Applying Theorem \ref{unrefined-unicity} again, we conclude that $\Xi$ must be an irreducible subrepresentation of $^g\sigma|_{G_y}$ for some $g\in N_G(G_x)$.\\

Projecting onto the reductive quotient $G_x/G_x^+$, we see that $\Xi$ is an irreducible subrepresentation of the restriction of $\sigma$ to the proper parabolic subgroup $\mathscr{P}=G_y/(G_y\cap G_x^+)$ of $G_x/G_x^+$. Let $\mathscr{P}$ have a standard Levi decomposition $\mathscr{P}=\mathscr{M}\mathscr{N}$, and let $\xi$ be an irreducible subrepresentation of the restriction to $\mathscr{M}$ of $\Xi$, so that $\xi$ has cuspidal support $(\mathscr{L},\zeta)$, say. Let $\mathscr{Q}$ denote the standard parabolic subgroup of $G_x/G_x^+$ with Levi factor $\mathscr{L}$, and let $\mathscr{Q}^\op$ denote the parabolic subgroup opposite to $\mathscr{Q}$. Forming the inverse image in $G_x$ of $\mathscr{Q}^\op$, we obtain a parahoric subgroup $G_w$ corresponding to some point $w\in\mathscr{B}(G)$. One obtains an identification $G_w/G_w^+=\mathscr{L}$, and the pair $(G_w,\zeta)$ is an unrefined depth-zero type. Now, just as in the proof of Theorem \ref{unrefined-unicity}, the space $\Hom_G(\cInd_{G_w}^G\ \zeta,\cInd_{G_y}^G\ \Xi)$ surjects onto the space $\Hom_{G_w\cap G_y}(\zeta,\Xi)$. The group $G_w\cap G_y$ is the inverse image in $G_x$ of $\mathscr{Q}^\op\cap\mathscr{P}$, which is precisely $\mathscr{L}$. Since $\zeta$ is the cuspidal support of the irreducible subrepresentation $\xi$ of $\Xi|_{\mathscr{M}}$, this latter space is certainly non-zero. So we see that $\cInd_{G_y}^G\ \Xi$ contains a non-cuspidal irreducible subquotient, which, by Frobenius reciprocity, is to say that $\Xi$ must be atypical.
\end{proof}

\section{Extension to archetypes}

So it remains for us to consider the impact of Theorem \ref{unrefined-unicity} once one performs an extension to refined depth-zero types. Any maximal compact subgroup of $G$ contains finitely many parahoric subgroups of $G$, but \emph{not every maximal compact subgroup must contain a maximal parahoric subgroup} --- for example, given a ramified quadratic extension $E/F$, the group $U(1,1)(E/F)$ contains a maximal compact subgroup of orthogonal type, the only parahoric subgroup in which is an Iwahori subgroup. Given a maximal compact subgroup $K$ of $G$, the maximal compact subgroup of the $G$-normalizer of the largest parahoric subgroup contained in $K$ coincides with $K$. We wish to see that if a $[G,\pi]_G$-type $\tau$ is defined on $K$, then $K$ must contain $G_x$ and $\tau$ must be isomorphic to a subrepresentation of $\Ind_{G_x}^K\ \sigma$ --- that is to say, we wish to see that the refined depth-zero types are precisely the archetypes for depth-zero supercuspidals.\\

In order to avoid worrying about conjugacy, let us adopt the convention that, when speaking of an archetype $(K,\tau)$, any parahoric subgroup contained in $K$ is standard (as is clearly possible).\\

We resume the notation of section \ref{section:parahoric-unicity}. In particular, $\pi$ is a depth-zero supercuspidal representation of $G$ containing an unrefined depth-zero type $(G_x,\sigma)$, with $x$ a vertex in the standard chamber of $\mathscr{B}(G)$.

\begin{lemma}\label{restriction-to-parahoric}
Let $\pi$ be a depth-zero supercuspidal representation of $G$, and let $(K,\tau)$ be a $[G,\pi]_G$-archetype. Then, up to $G$-conjugacy, $K$ contains the maximal parahoric subgroup $G_x$ of $G$ on which the unrefined depth-zero type for $\pi$ is defined, and $\tau|_{G_x}$ is isomorphic to a sum of unrefined depth-zero types, which are pairwise $K$-conjugate.
\end{lemma}

\begin{proof}
Let $y$ be such that $G_y$ is the largest parahoric subgroup contained in $K$, so that $K$ is the maximal compact subgroup of the normalizer of $G_y$ (note that $y$ will be a vertex if and only if $K$ normalizes the parahoric subgroup corresponding to some vertex). Since $\tau\hookrightarrow\pi|_K$, we have that $\tau|_{G_y}\hookrightarrow\pi|_{G_y}$, and we have seen by Proposition \ref{all-parahorics-unicity} that any irreducible subrepresentation of this latter representation which is not an unrefined depth-zero type must be contained in an irreducible depth-zero non-cuspidal representation of $G$. So pick an irreducible subrepresentation $\rho$ of $\tau|_{G_y}$, and suppose for contradiction that $\rho$ is contained in such a non-cuspidal representation $\pi'$. The representation $\pi'$ contains an unrefined depth-zero type $(G_w,\sigma')$ with $G_w$ a non-maximal standard parahoric subgroup of $G$; let $R$ denote the subrepresentation of $\cInd_{G_w}^G\ \sigma'$ generated by $\rho$. Since $\rho$ is irreducible, it is generated by a single vector $v$, say. By Frobenius reciprocity, the representation $\cInd_{G_w}^G\ \sigma$ contains a canonical copy of $\rho$, and we see that $R$ coincides with the subrepresentation of $\cInd_{G_w}^G\ \sigma'$ generated by $v$.\\

As $R$ is finitely generated, it admits an irreducible quotient $\Psi$, say. Since $R$ is contained in the category $\Rep^{\SS_{\sigma'}}(G)$, which is closed under taking quotients, we see that $\Psi$ contains $\sigma'$, and so $\Psi$ has a vector fixed by $G_w^+$. Now, as $\Psi$ is a quotient of the representation $R$, which we generated by $\rho$, and $\Psi$ has a vector fixed by $G_w^+$, we deduce that $\rho$ has a vector fixed by $G_y\cap G_w^+$. In particular, $\tau$ must also have a vector fixed by $G_y\cap G_w^+$. We claim that, given any point $w\in\mathscr{B}(G)$ and any point $y\in\mathscr{B}(G)$ contained in the interior of a facet of higher dimension than $y$, there exists a $g\in G$ such that $G_w^+\subset G_{gy}$. Indeed, $G_w^+$ is contained in the pro-unipotent radical of an Iwahori subgroup, and hence in the Iwahori subgroup itself. There is an element of the orbit of $y$ in the chamber corresponding to this Iwahori subgroup, and the claim follows. So we may conjugate our choice of depth-zero type $(G_w,\sigma')$ and assume without loss of generality that $G_w^+\subset G_y$. Hence $\tau$ has a non-zero vector fixed by $G_w^+$, and so there exists an irreducible subquotient of $\cInd_K^G\ \tau$ with such a vector. On the other hand, $(K,\tau)$ is a $[G,\pi]_G$-type, and so any such subquotient must be supercuspidal. But a supercuspidal representation may not possess a non-zero vector fixed by the pro-unipotent radical of any non-maximal parahoric subgroup of $G$.\\

In order to avoid this contradiction, we conclude that the restriction to $G_y$ of $\tau$ is a sum of unrefined depth-zero types. So by Proposition \ref{all-parahorics-unicity}, $y$ must be a vertex, and moreover, by Theorem \ref{unrefined-unicity}, $y$ must be conjugate to $x$ under the action of $G$. The result follows, with the fact that the components of $\tau|_{G_x}$ are pairwise $K$-conjugate following by Clifford theory.
\end{proof}

With this in place, we come to our main result:

\begin{theorem}[The unicity of types for depth-zero supercuspidals]\label{refined-unicity}
Let $\pi$ be a depth-zero supercuspidal representation of $G$.Then there exists a unique $[G,\pi]_G$-archetype $(K,\tau)$. The group $K$ contains a maximal parahoric subgroup $G_x$ of $G$, and the restriction to $G_x$ of $\tau$ is isomorphic to a sum of unrefined depth-zero types contained in $\pi$.
\end{theorem}

\begin{proof}
The existence of such an archetype is Theorem \ref{morris-refinement}.\\So let $(K,\tau)$ be such a $[G,\pi]_G$-archetype, and let $(G_x,\sigma)$ be an unrefined depth-zero type contained in $\pi$, with $G_x$ standard. We have seen that, without loss of generality, we may assume that $K$ contains $G_x$ and $\tau|_{G_x}$ is isomorphic to a sum of conjugates of $\sigma$ (possibly with multiplicity). So the $[G,\pi]_G$-archetypes are exhausted by the $[G,\pi]_G$-typical subrepresentations of $\Ind_{G_x}^K\ \sigma$. There is a unique conjugacy class of such representations.\\

Indeed, let $\tau'$ be another such representation, which, without loss of generality, we assume to be a representation of $K$. Since both $\tau$ and $\tau'$ are $[G,\pi]_G$-types, $\pi$ arises as a subquotient of both $\cInd_K^G\ \tau$ and $\cInd_K^G\ \tau'$, so that $\tau$ and $\tau'$ intertwine in $G$, i.e. there exists a $g\in G$ such that $\Hom_{K\cap{}^gK}(\tau,{}^g\tau')\neq 0$. Certainly $\tau|_{G_x}$ must then also intertwine with $\tau'|_{G_x}$, which, by Frobenius reciprocity, implies that that there exists a $g\in G$ such that
\[0\neq\Hom_{G_x}(\Res_{G_x}^K\ \tau,\Ind_{^gG_x\cap G_x}^{G_x}\Res_{^gG_x\cap G_x}^{^gG_x}\ ^g\tau').
\]
We have seen that the restriction to $G_x$ of $\tau$ is a sum of unrefined depth-zero types, say $\tau|_{G_x}=\bigoplus_h{}^h\sigma^{\oplus m(\sigma)}$. On the other hand, we know by Theorem \ref{unrefined-unicity} that any subrepresentation of
\[\Ind_{^gG_x\cap G_x}^{G_x}\Res_{^gG_x\cap G_x}^{^gG_x}\ ^g\tau'=\Ind_{^gG_x\cap G_x}^{G_x}\Res_{^gG_x\cap G_x}^{^gG_x}\ ^g\left(\bigoplus_{K/N_K(\sigma)}{}^h\sigma^{\oplus m(\sigma)}\right)
\]
such that $g\not\in N_G(G_x)$ must be atypical. So $g\in N_G(G_x)$. Since $^gK$ contains $^gG_x=G_x$ and $K$ is the unique maximal compact subgroup in which $G_x$ is contained, we must have $^gK=K$, so that $\tau'\simeq{}^g\tau$ for some $g\in N_G(K)$.
\end{proof}

In particular, we have established Conjecture \ref{unicity-conjecture} in the case of a depth-zero supercuspidal representation of an arbitrary group.

\section{The Langlands correspondence of DeBacker and Reeder}\label{DeBacker--Reeder-section}

Having completely described the unicity of types for depth-zero supercuspidals, we now turn to the question of how this fits in with the local Langlands correspondence. In order to be able to do this, we must restrict to the case where a local Langlands correspondence has been established for the depth-zero supercuspidal representations. Recall that an inner form of a  connected reductive $F$-group $\HHH$ is a connected reductive $F$-group $\HHH'$ with $\HHH\times_F\bar{F}\simeq\HHH'\times_G\bar{F}$; we extend this definition in to natural way to the $p$-adic groups of $F$-points. The inner forms of $G$ are naturally parametrized by the image in $H^1(\Gal(\bar{F}/F),\mathrm{Aut}(\GGG))$ of the Galois cohomology group $H^1(\Gal(\bar{F}/F),\GGG^\ad)$, where $\GGG^\ad=\GGG/\ZZZ_\GGG$ is the adjoint form of $\GGG$. There is then a canonical map $H^1(\Gal(\bar{F}/F),\GGG)\rightarrow H^1(\Gal(\bar{F}/F),\GGG^\ad)$, which is in general neither injective or surjective; we say that an inner form of $\GGG$ is \emph{pure} if it corresponds to a cohomology class in the image of this map. For the remainder of the paper, we impose the following hypothesis:

\begin{hypothesis}\label{unramified-hypothesis}
We assume that $\GGG$ is a pure inner form of an unramified group, i.e. there exists a pure inner form of $\GGG$ which is $F$-quasi-split and $E$-split for some finite unramified extension $E/F$.
\end{hypothesis}

This assumption is precisely that which is required for us to have a suitable Langlands correspondence for (a certain subset of) the depth-zero supercuspidal representations of $G$. We begin by recalling some generalities on the local Langlands correspondence.\\

Denote by $W_F$ the Weil group of $F$, i.e. the locally profinite subgroup of $\Gal(\bar{F}/F)$ generated by the inertia group $I_F=\Gal(\bar{F}/F^\ur)$ and some fixed choice $\Frob$ of Frobenius element. Let $I_F^+$ denote the wild inertia group, i.e. the maximal pro-$p$ open normal subgroup of $I_F$. The \emph{tame inertia group} is then the quotient $I_F^\tame=I_F/I_F^+$.\\

The $F$-group $\GGG$ is classified by its based root datum, together with the action of $W_F$ on this datum. Under the natural duality on root data induced from the root-coroot duality, $\GGG$ then corresponds to a unique connected reductive algebraic group $\hat{\GGG}$ defined over $\C$. Moreover, the natural action of $W_F$ on the root datum of $\GGG$ then defines an action of $W_F$ on that of $\hat{\GGG}$, and hence on $\hat{\GGG}$ itself. The \emph{Langlands dual group} of $G$ is then defined to be the semidirect product $^L\GGG=W_F\ltimes\hat{\GGG}$. Since the action of $W_F$ will always factor through a finite quotient, there is an identification of $^L\GGG$ with a reductive group defined over $\C$, the component group of which is finite; we will implicitly identify $^L\GGG$ with a group of this form throughout the following discussion. An \emph{$L$-parameter} for $G$ is then a homomorphism $W_F\times\SL2(\C)\rightarrow{}^L\GGG$ which is continuous upon restriction to $W_F$, algebraic upon restriction to $\SL2(\C)$, and such that the image of $\Frob$ is a semisimple element. Say that two $L$-parameters are equivalent if they are conjugate in $^L\GGG$, and denote by $\Ll(G)$ the set of equivalence classes of $L$-parameters for $G$. The local Langlands conjectures then predict that there should exist a unique surjective, finite-to-one map $\Irr(G)\rightarrow\Ll(G)$ satisfying a number of natural properties. The fibres of this map then partition $\Irr(G)$ into finite sets, which we call $L$-\emph{packets}.\\

We will be interested in certain classes of $L$-parameters. Say that an $L$-parameter $\varphi$ is \emph{discrete} if its image is not contained in any proper Levi subgroup of $^L\GGG$, and \emph{tame} if the restriction of $\varphi$ to $I_F^+$ is the trivial homomorphism. We note that while an $L$-packet should contain a discrete series representation of $G$ if and only if it contains only discrete series representations (which should occur precisely when the corresponding $L$-parameter is discrete). We emphasize it is \emph{not} the case that an $L$-packet containing a supercuspidal representation should consist only of supercuspidal representations. On the other hand, however, the elements of an $L$-packet should all be of the same depth, and an $L$-packet should contain a depth-zero representation if and only if the corresponding $L$-parameter is tame. If $\varphi$ is a discrete parameter, then we say that $\varphi$ is \emph{regular} if the action of $\SL2(\C)$ is trivial; thus we can (and will) identify regular $L$-parameters with homomorphisms $W_F\rightarrow{}^L\GGG$.\\

In \cite{debacker2009depthzero}, for a group $G$ satisfying Hypothesis \ref{unramified-hypothesis}, DeBacker and Reeder associate a finite $L$-packet of depth zero supercuspidal representations to a large number of the tame discrete $L$-parameters for $G$. We now give a  brief recap of the relevant parts of this construction.\\

We will be interested in \emph{tame regular semisimple elliptic $L$-parameters} (or TRSELPs for short). These are tame regular parameters of a specific form. Since $W_F$ is topologically generated by $I_F$ and $\Frob$, a tame regular parameter $\varphi:W_F\rightarrow{}^L\GGG$ is completely determined by the data of a homomorphism $I_F^\tame\rightarrow{}^L\GGG$ representing the action of $I_F$ (which factors through $I_F^\tame$ since $\varphi$ is tame) and an element $f\in{}^L\GGG$ representing the action of $\Frob$.

\begin{definition}\label{TRSELP}
A \emph{tame regular semisimple elliptic $L$-parameter} (TRSELP) is a pair $\varphi=(s,f)$ consisting of:
\begin{enumerate}[(i)]
\item a continuous homomorphism $s:I_F^\tame\rightarrow\hat{\TTT}$ for some maximal torus $\hat{\TTT}$ in $^L\GGG$ satisfying $C_{^L\GGG}(\hat{\TTT})=\hat{\TTT}$; and
\item an element $f\in\hat{N}=N_{^L\GGG}(\hat{\TTT})$ satisfying certain conditions which will not be important for our purposes; see \cite[Section 4.1]{debacker2009depthzero} for the details.
\end{enumerate}
We denote by $\Ll_{\tame,\reg}(G)$ the set of equivalence classes of TRSELPs.
\end{definition}

For the remainder of the paper, since we will need to define a large number of classes of objects via TRSELPs, for readability we will abuse terminology slightly and drop the adjective ``semisimple elliptic'' from our notation. In the case that the centre of $\GGG$ is connected, the TRSELPs should be precisely those tame regular parameters corresponding to $L$-packets which consist \emph{only} of depth-zero supercuspidal representations. In the case that $\GGG$ does not have a connected centre, then the TRSELPs should be the tame regular parameters corresponding to $L$-packets consisting of supercuspidal representations which are generic in some sense --- for example, when $F$ is of odd residual characteristic, then the $L$-packets in $\Irr(\SL2(F))$ which contain depth-zero supercuspidals generically contain two elements; there exists a unique such packet containing four elements (corresponding to the unique representation of $\GL2(\kkk_F)$ which restricts reducibly to $\SL2(\kkk_F)$), the parameter of which is \emph{not} a TRSELP.

\begin{definition}
A \emph{tame regular inertial type} for $G$ is the restriction to $I_F$ of an element of $\Ll_{\tame,\reg}(G)$. Note that this is equivalent to defining a tame regular inertial type to be a homomorphism $s:I_F^\tame\rightarrow{}^L\GGG$ as in Definition \ref{TRSELP}. Say that two tame regular inertial types $s,s'$ are equivalent if there exist choices $f,f'$ such that the TRSELPs $(s,f)$ and $(s',f')$ are equivalent, and denote by $\Ii_{\tame,\reg}(G)$ the set of equivalence classes of tame regular inertial types.
\end{definition}

In particular, a TRSELP consists precisely of the data of an underlying tame regular inertial type together with a compatible action of $\Frob$; thus we get a well-defined surjective map $\Res_{I_F}^{W_F}:\Ll_{\tame,\reg}(G)\rightarrow\Ii_{\tame,\reg}(G)$. With this in place, we are ready to sketch out the construction of the $L$-packet associated to a TRSELP. Fix a choice $\varphi=(s,f)\in\Ll_{\tame,\reg}(G)$ of TRSELP. Let $\TTT\subset\GGG$ be the maximal split torus, and write $X=X_*(\TTT)$. Let $\hat{\vartheta}$ denote the automorphism of $^L\GGG$ arising from the action of $\Frob$ (via the action of $W_F$ on $^L\GGG$ inherited from the action of $W_F$ on the root datum of $\GGG$). This automorphism $\hat{\vartheta}$ then gives a dual automorphism $\vartheta$ of $X$. Moreover, upon restriction to $\hat{\TTT}$, the element $\varphi(\Frob)$ of $^L\GGG$ normalizes $\hat{\TTT}$ and acts by an element of the form $\hat{\vartheta}\hat{w}$, for some $\hat{w}$ in the Weyl group of $\hat{\TTT}$. This element $\hat{w}$ then also gives a dual automorphism $w$ of $X$. We denote by $X_w$ the pre-image in $X$ of $[X/(1-w\vartheta)X]_{\mathrm{tors}}$.\\

Now fix a choice of $\lambda\in X_w$. To $\lambda$, one may associate a certain 1-cocycle $u_\lambda$ \cite[Section 2.7]{debacker2009depthzero}; the twisted Frobenius $F_\lambda=\mathrm{Ad}(u_\lambda)\circ\Frob$ acts on the apartment $\mathscr{A}(G,\TTT)$ and stabilizes a unique vertex $x_\lambda$; we therefore obtain for each $\lambda\in X_w$ a maximal parahoric subgroup $G_\lambda$ of $G$.\\

Moreover, by the local Langlands correspondence for tori (which is well-known, but reproved in the depth zero setting in \cite[Section 4.3]{debacker2009depthzero}), the homomorphism $s:I_F^\tame\rightarrow\hat{\TTT}$ corresponds to a character of $\TTT(F)$. DeBacker and Reeder associate to $(\varphi,\lambda)$ a group $T_\lambda=\TTT_\lambda(F)$ of $F$-points of a certain conjugate $\TTT_\lambda$ of $\TTT$, and hence a character $\theta_\lambda$ of $T_\lambda$. This character will be of depth-zero in the sense that it will be trivial on $T_\lambda\cap G_\lambda^+$, but non-trivial on $T_\lambda\cap G_\lambda$. In particular, we may identify $\theta_\lambda$ with a character of the torus in the reductive quotient $G_\lambda/G_\lambda^+$ obtained as the image of $T_\lambda\cap G_\lambda$. Deligne---Lusztig theory then gives a virtual representation $R_{T_\lambda}^{\theta_\lambda}$ of $G_\lambda/G_\lambda^+$. At this point, the ``genericity'' property of a TRSELP (as opposed to an arbitrary tame regular $L$-parameter) guarantees that the character $\theta_\lambda$ is in general position, so that one of $\pm R_{T_\lambda}^{\theta_\lambda}$ will be an irreducible cuspidal representation $\sigma_\lambda$ of $G_\lambda/G_\lambda^+$. We therefore obtain an unrefined depth-zero type $(G_\lambda,\sigma_\lambda)$.\\

At this point, the elements of $\Irr(G)$ which arise as subquotients of $\cInd_{G_\lambda}^G\ \sigma_\lambda$ for some TRSELP $\varphi$ and some $\lambda\in X_w$ are grouped into $L$-packets $\Pi(\varphi)$ according to, in particular, the following rules:
\begin{enumerate}[(i)]
\item Every element $\lambda$ of $X_w$ determines an unrefined depth zero type $(G_\lambda,\sigma_\lambda)$ in some pure inner form of $G$ and, as $\lambda$ ranges over the subset of $X_w$ consisting of those $\lambda$ which give rise to an unrefined depth zero type in $G$, for each $\lambda$ there exists a unique element $\pi_\lambda$ of $\Pi(\varphi)$ containing the unrefined depth zero type $(G_\lambda,\sigma_\lambda)$. In other words, the elements of $\Pi(\varphi)$ are precisely a choice of irreducible subquotient of $\cInd_{G_\lambda}^G\ \sigma_\lambda$ for each $\lambda$; this choice is determined by $f$, and will turn out to be irrelevant for our purposes.
\item As $\varphi$ varies over all TRSELPs of a given inertial type, the union of the packets $\Pi(\varphi)$ is equal to the set of irreducible subquotients of the representation $\cInd_{G_\lambda}^G\ \sigma_\lambda$, as $\lambda$ varies. Moreover, if $f,f'$ are such that $\Pi(s,f)\cap\Pi(s,f')\neq\emptyset$, then the TRSELPs $(s,f)$ and $(s,f')$ are equivalent.
\end{enumerate}
Moreover, DeBacker and Reeder show that this process results in $L$-packets which are \emph{stable} in a technical, character-theoretic sense, and that these are the smallest such stable $L$-packets. Thus, while they do not show that the resulting correspondence satisfies all of the conditions which are expected of the local Langlands correspondence, it is extremely likely that it is the correct such correspondence.

\begin{hypothesis}
We assume that the assignment of an $L$-packet of depth-zero supercuspidals to each TRSELP given by DeBacker and Reeder satisfies all of the expected properties of the local Langlands correspondence.
\end{hypothesis}

\begin{definition}
Say that a representation of $G$ is \emph{tame regular (semisimple elliptic)} if it is contained in $\Pi(\varphi)$ for some TRSELP $\varphi$. Denote by $\Irr_{\tame,\reg}(G)$ the set of isomorphism classes of tame regular representations of $G$.
\end{definition}
Note that a tame regular representation is necessarily an irreducible depth-zero supercuspidal representation of $G$.\\

We realize the DeBacker--Reeder construction as a surjective, finite-to-one map $\rec:\Irr_{\tame,\reg}(G)\rightarrow\Ll_{\tame,\reg}(G)$ which assigns to each $\pi\in\Irr_{\tame,\reg}(G)$ the unique $\varphi\in\Ll_{\tame,\reg}(G)$ such that $\pi$ is contained in $\Pi(\varphi)$. As an immediate consequence of properties (i) and (ii) above of the $L$-packets $\Pi(\varphi)$, we immediately obtain the following:

\begin{lemma}\label{restriction-to-inertia-TRSELP}
Let $\pi,\pi'\in\Irr_{\tame,\reg}(G)$. Then $\rec(\pi)$ and $\rec(\pi')$ are of the same inertial type if and only if $\pi$ and $\pi'$ contain a common unrefined depth-zero type.
\end{lemma}

\section{The tame inertial Langlands correspondence}

Having described the local Langlands correspondence for $\Irr_{\tame,\reg}(G)$, we describe a similar correspondence on the level of types and inertial types.\\

Denote by $\Dd_{\tame,\reg}(G)$ the set of conjugacy classes of unrefined depth-zero types which are contained in some element of $\Irr_{\tame,\reg}(G)$, and by $\Aa_{\tame,\reg}(G)$ the set of $[G,\pi]_G$-archetypes, as $\pi$ ranges through $\Irr_{\tame,\reg}(G)$. Since, by Theorem \ref{unrefined-unicity}, each such $[G,\pi]_G$-archetype $(K,\tau)$ restricts to the unique conjugacy class of maximal parahoric subgroups contained in the conjugacy class of $K$ as a direct sum of pairwise $G$-conjugate unrefined depth-zero types, there is a canonical surjective map $\Aa_{\tame,\reg}(G)\rightarrow\Dd_{\tame,\reg}(G)$. Moreover, since $\pi\in\Irr_{\tame,\reg}(G)$ contains a unique element of $\Dd_{\tame,\reg}(G)$ and a unique element of $\Aa_{\tame,\reg}(G)$, there are also canonical surjective maps $T:\Irr_{\tame,\reg}(G)\rightarrow\Aa_{\tame,\reg}(G)$ and $D:\Irr_{\tame,\reg}(G)\rightarrow\Dd_{\tame,\reg}(G)$; it is clear that $D$ factors through $T$ via canonical map $\Aa_{\tame,\reg}(G)\rightarrow\Dd_{\tame,\reg}(G)$.

\begin{theorem}[The tame inertial Langlands correspondence]\label{tame-inertial-correspondence}
Let $G$ be a pure inner form of an unramified $p$-adic group.
\begin{enumerate}[(i)]
\item There exists a unique surjective, finite-to-one map $\inerP:\Dd_{\tame,\reg}(G)\rightarrow\Ii_{\tame,\reg}(G)$ such that the following diagram commutes:
\[\xymatrix{
\Irr_{\tame,\reg}(G)\ar[r]^-\rec\ar[d]_D & \Ll_{\tame,\reg}(G)\ar[d]^{\Res_{I_F}^{W_F}}\\
\Dd_{\tame,\reg}(G)\ar[r]_-\inerP & \Ii_{\tame,\reg}(G)
}\]
Given $\varphi\in\Ll_{\tame,\reg}(G)$, one has $\#\inerP^{-1}(\varphi|_{I_F})=\#\rec^{-1}(\varphi)$.
\item There exists a unique surjective, finite-to-one map $\iner:\Aa_{\tame,\reg}(G)\rightarrow\Ii_{\tame,\reg}(G)$ such that the following diagram commutes;
\[\xymatrix{
\Irr_{\tame,\reg}(G)\ar[r]^-\rec\ar[d]_T & \Ll_{\tame,\reg}(G)\ar[d]^{\Res_{I_F}^{W_F}}\\
\Aa_{\tame,\reg}(G)\ar[r]_-\iner & \Ii_{\tame,\reg}(G)
}\]
The map $\iner$ factors uniquely through $\inerP$ via the canonical map $\Aa_{\tame,\reg}(G)\rightarrow\Dd_{\tame,\reg}(G)$ and, given $\varphi\in\Ll_{\tame,\reg}(G)$, one has
\[\#\iner^{-1}(\varphi|_{I_F})=\sum_{(G_x,\sigma)\in\inerP^{-1}(\varphi|_{I_F})	}\#\SS_\sigma.
\]
\end{enumerate}
\end{theorem}

\begin{proof}
Let $R:\Aa_{\tame,\reg}(G)\rightarrow\Irr_{\tame,\reg}(G)$ be \emph{any} map which, to an archetype $(K,\tau)$, assigns an irreducible subquotient of $\cInd_K^G\ \tau$; we then define $\iner=\Res_{I_F}^{W_F}\circ\rec\circ R$. This is well-defined: since $(K,\tau)$ is a $[G,\pi]_G$-archetype for some $\pi\in\Irr_{\tame,\reg}(G)$, any two subquotients of $\cInd_K^G\ \tau$ are unramified twists of one another; however, given an unramified character $\omega$ of $G$, restricting to $I_F$ induces an isomorphism between $\rec(\pi)|_{I_F}$ and $\rec(\pi\otimes\omega)|_{I_F}$. It follows immediately that $\iner$ is the unique map $\Aa_{\tame,\reg}(G)\rightarrow\Ii_{\tame,\reg}(G)$ such that the diagram in \emph{(ii)} commutes. Since $\rec$ and $\Res_{I_F}^{W_F}$ are surjective, it follows that $\iner$ is also surjective.\\

We now establish \emph{(i)}. Fix a map $A:\Dd_{\tame,\reg}(G)\rightarrow\Aa_{\tame,\reg}(G)$ which, to each unrefined depth-zero type $(G_x,\sigma)$, assigns an irreducible subrepresentation of $\Ind_{G_x}^K\ \sigma$, where $K$ denotes the maximal compact subgroup of $N_G(G_x)$. Then we define $\inerP:\Dd_{\tame,\reg}(G)\rightarrow\Ii_{\tame,\reg}(G)$ by setting $\inerP=\iner\circ A$. We must first check that $\inerP$ is well-defined. Let $A':\Dd_{\tame,\reg}(G)\rightarrow\Aa_{\tame,\reg}(G)$ be another such map, and let $\inerP'=\iner\circ A'$. Then we have a diagram
\[\xymatrix{
& \Irr_{\tame,\reg}(G)\ar[r]^-\rec\ar[d]_T\ar@/_2pc/[ddl]_D & \Ll_{\tame,\reg}(G)\ar[d]^{\Res_{I_F}^{W_F}}\\
& \Aa_{\tame,\reg}(G)\ar[r]_-\iner & \Ii_{\tame,\reg}(G)\\
\Dd_{\tame,\reg}(G)\ar@/^/[ru]^A\ar@/_/[ru]_{A'}\ar@/_1pc/[rru]^{\inerP'}\ar@/_2pc/[rru]_{\inerP}
}\]
We have seen that the top square and the leftmost triangle commute, while in the bottom triangle $\inerP=\iner\circ A$ and $\inerP'=\iner\circ A'$ by definition. Suppose that $\inerP\neq\inerP'$, and pick $(G_x,\sigma)\in\Dd_{\tame,\reg}(G)$ such that $\inerP(G_x,\sigma)\neq\inerP'(G_x,\sigma)$. Hence $A(G_x,\sigma)$ and $A'(G_x,\sigma)$ must lie in different fibres of $\iner$, which is to say that they must lie in different fibres of $\Res_{I_F}^{W_F}\circ\rec\circ R$. This means that $\cInd_K^G\ A(G_x,\sigma)$ and $\cInd_K^G\ A'(G_x,\sigma)$ admit irreducible subquotients lying in different fibres of $\Res_{I_F}^{W_F}\circ\rec$. Hence there exist irreducible subquotients $\rho,\rho'$, say, of $\cInd_{G_x}^G\ \sigma$ which lie in different fibres of $\Res_{I_F}^{W_F}\circ\rec$. By Lemma \ref{restriction-to-inertia-TRSELP} this is not the case, and so $\inerP$ is well-defined.\\

Since $\inerP$ is well-defined, $\iner$ factors uniquely through $\inerP$ via the canonical map $\Aa_{\tame,\reg}(G)\rightarrow\Dd_{\tame,\reg}(G)$. Let $\varphi\in\Ll_{\tame,\reg}(G)$. We claim that the fibres $\inerP^{-1}(\varphi|_{I_F})$ and $\rec^{-1}(\varphi)$ are in canonical bijection. There is a canonical injective map between these two sets. Indeed, given $(G_x,\sigma)\in\inerP^{-1}(\varphi|_{I_F})$, there exists a unique element of $\rec^{-1}(\varphi)$ containing $(G_x,\sigma)$, which can be observed by combining \cite[Lemma 4.5.2 (3)]{debacker2009depthzero} and \cite[Theorem 4.5.3]{debacker2009depthzero}. This gives a map $\inerP^{-1}(\varphi|_{I_F})\rightarrow\rec^{-1}(\varphi)$, which is injective since a depth-zero irreducible representation of $G$ contains a unique conjugacy class of unrefined depth-zero types. Moreover, this map is clearly seen to be surjective: given a TRSELP $\varphi$ and an element $\pi$ of $\rec^{-1}(\varphi)$, we have noted that $\pi$ contains an unrefined depth zero type which, by the commutativity of the above diagram, is contained in $\iner_\Dd^{-1}(\varphi|_{I_F})$.\\

Finally, since $\iner$ factors uniquely through $\inerP$, it follows that the elements of the fibre $\iner^{-1}(\varphi|_{I_F})$ are precisely the irreducible subrepresentations of the representations $\Ind_{G_x}^{K_x}\ \sigma$ (modulo $G$-conjugacy), where $(G_x,\sigma)$ ranges through the elements of $\inerP^{-1}(\varphi|_{I_F})$ and $K_x$ denotes the maximal compact subgroup of $N_G(G_x)$. Each subrepresentation $\tau$ of some $\Ind_{G_x}^{K_x}\ \sigma$ is an $\frak{s}$-type, for some $\frak{s}\in\SS_\sigma$. We have already seen that there is a unique $\frak{s}$-type (up to conjugacy) for each $\frak{s}$; it follows that
\[\#\iner^{-1}(\varphi|_{I_F})=\sum_{(G_x,\sigma)\in\inerP^{-1}(\varphi|_{I_F})}\#\SS_\sigma,
\]
as desired.
\end{proof}

\begin{remarks}
\begin{enumerate}[(i)]
\item In \cite{macdonald1980zetafunctions}, Macdonald constructs a ``Langlands correspondence'' for certain finite reductive groups, via $p$-adic methods. In particular, this provides --- for certain groups --- a correspondence between cuspidal irreducible representations of finite reductive groups and certain homomorphisms from the inertia group to complex tori. One interpretation of the above result is that it shows that the DeBacker--Reeder construction subsumes Macdonald's result, providing via the map $\inerP$ a similar correspondence.
\item In the case that the maximal parahoric subgroups of $G$ are all maximal as compact subgroups of $G$ (as happens, for example, when $\GGG$ is semisimple and simply connected) the sets $\Aa_{\tame,\reg}(G)$ and $\Dd_{\tame,\reg}(G)$ and the maps $\iner$ and $\inerP$ coincide; in particular, in this case the description of the inertial correspondence consists simply of the statement \emph{(i)} above.
\item The reader should note that the definitions of $\iner$ and $\inerP$ do \emph{not} rely on our unicity results: one  can see that these are the unique well-defined surjective maps making the appropriate diagrams commute without knowing anything about unicity. However, the real strength of the above result is in the description of the fibres of these maps, for which our unicity results are crucial.
\item An undesirable aspect of our result is that we only obtain a description for $\Irr_{\tame,\reg}(G)$. There should exist a more general correspondence, but there are a number of obstacles preventing the proof of such a result. Firstly, and most significantly, we do not yet have a construction of the local Langlands correspondence for arbitrary depth-zero irreducible representations of $G$; in particular, this prevents the observation that any map analogous to $\inerP$ is well-defined from being made. On top of this, there are two further serious complications. Firstly, the relationship between depth-zero types and archetypes becomes far more complicated, meaning that describing the fibres of $\iner$ via such a simple formula is unlikely to be possible. For example, in $\GL2(F)$, one already sees that while the Steinberg representation contains a unique unrefined depth-zero type (the trivial representation of the Iwahori subgroup), it contains two archetypes: the trivial representation of $\GL2(\dedekind)$, and the inflation to $\GL2(\dedekind)$ of the Steinberg representation of $\GL2(\kkk)$. Finally, there are additional complications in the proof of unicity which would require methods distinct to those employed in this paper. It is likely that our methods would suffice only to show that any typical representation of a maximal parahoric subgroup of $G$ is contained in a certain infinite length representation corresponding to the trivial Mackey summand here. Once one considers representations with supercuspidal support defined on a proper parabolic subgroup, this summand is no longer irreducible.
\end{enumerate}
\end{remarks}

\bibliographystyle{amsalpha}
\addcontentsline{toc}{chapter}{Bibliography}
\bibliography{depthzero}

\end{document}